\documentclass[11pt,a4paper]{amsart}
\usepackage{graphicx,multirow,array,amsmath,amssymb}

\newtheorem{theorem}{Theorem}

\newtheorem{lemma}[theorem]{Lemma}
\newtheorem{corollary}[theorem]{Corollary}

\begin{document}

\title{Small knot mosaics and partition matrices}

\author[K. Hong]{Kyungpyo Hong}
\address{Department of Mathematics, Korea University, Anam-dong, Sungbuk-ku, Seoul 136-701, Korea}
\email{cguyhbjm@korea.ac.kr}
\author[H. Lee]{Ho Lee}
\address{Department of Mathematical Sciences, KAIST, 291 Daehak-ro, Yuseong-gu, Daejeon 305-701, Korea}
\email{figure8@kaist.ac.kr}
\author[H. J. Lee]{Hwa Jeong Lee}
\address{Department of Mathematical Sciences, KAIST, 291 Daehak-ro, Yuseong-gu, Daejeon 305-701, Korea}
\email{hjwith@kaist.ac.kr}
\author[S. Oh]{Seungsang Oh}
\address{Department of Mathematics, Korea University, Anam-dong, Sungbuk-ku, Seoul 136-701, Korea}
\email{seungsang@korea.ac.kr}

\thanks{The corresponding author(Seungsang Oh) was supported by Basic Science Research Program through
the National Research Foundation of Korea(NRF) funded by the Ministry of Science,
ICT \& Future Planning(MSIP) (No.~2011-0021795).}
\thanks{This work was supported by the National Research Foundation of Korea(NRF) grant
funded by the Korea government(MEST) (No. 2011-0027989).}

\begin{abstract}
Lomonaco and Kauffman introduced knot mosaic system to give a definition of quantum knot system.
This definition is intended to represent an actual physical quantum system.
A knot $(m,n)$-mosaic is an $m \times n$ matrix of mosaic tiles which are $T_0$ through $T_{10}$ depicted as below,
representing a knot or a link by adjoining properly that is called suitably connected.
An interesting question in studying mosaic theory is how many knot $(m,n)$-mosaics are there.
$D_{m,n}$ denotes the total number of all knot $(m,n)$-mosaics.
This counting is very important because
the total number of knot mosaics is indeed the dimension of the Hilbert space of these quantum knot mosaics.

In this paper, we find a table of the precise values of $D_{m,n}$ for $4 \leq m \leq n \leq 6$ as below.
Mainly we use a partition matrix argument
which turns out to be remarkably efficient to count small knot mosaics. \\

\begin{center}
\begin{tabular}{|c|r|r|r|}   \hline
$D_{m,n}$ & $n=4$ & $n=5$ & $n=6$ \\    \hline
$m=4$ & $2594$ & $54,226$ & $1,144,526$ \\    \hline
$m=5$ & & $4,183,954$ & $331,745,962$ \\    \hline
$m=6$ & & & $101,393,411,126$ \\   \hline
\end{tabular}
\end{center}

\end{abstract}

\maketitle

\section{Introduction} \label{sec:intro}
The connection between knots and quantum physics has been of great interest.
One of remarkable discovery in the theory of knots  is the Jones polynomial,
and it turned out that the explanation of the Jones polynomial has to do with quantum theory.
The readers refer \cite{J1, J2, K1, K2, L, LK2, SJ}.
Lomonaco and Kauffman introduced a knot mosaic system to set the foundation
for a quantum knot system in the series of papers \cite{LK1, LK3, LK4, LK5}.
Their definition of quantum knots was based on the planar projections of knots and the Reidemeister moves.
They model the topological information in a knot by a state vector in a Hilbert space
that is directly constructed from knot mosaics.
They proposed several questions in \cite{LK3},
and this paper aims to answer to one of them.

Throughout this paper the term ``knot" means either a knot or a link.
We begin by introducing the basic notion of knot mosaics.
Let $\mathbb{T}$ denote the set of the following eleven symbols which are called {\em mosaic tiles\/};

\begin{figure}[ht]
\includegraphics[scale=1]{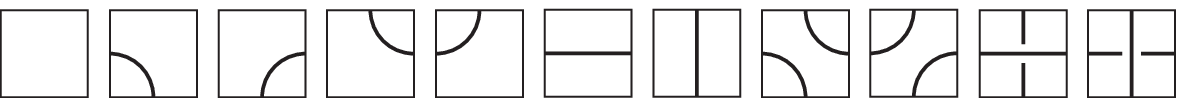}
\[ \ T_0 \hspace{7mm} T_1 \hspace{7mm} T_2 \hspace{7mm} T_3 \hspace{7mm} T_4 \hspace{7mm}
T_5 \hspace{7mm} T_6 \hspace{7mm} T_7 \hspace{7mm} T_8 \hspace{7mm} T_9 \hspace{7mm} T_{10} \]
\vspace{-8mm}
\label{fi1}
\end{figure}

For positive integers $m$ and $n$, we define an {\em $(m,n)$-mosaic\/}
as an $m \times n$ matrix $M=(M_{ij})$ of mosaic tiles.
We denote the set of all $(m,n)$-mosaics by $\mathbb{M}^{(m,n)}$.
Obviously $\mathbb{M}^{(m,n)}$ has $11^{mn}$ elements.
Indeed this rectangular version of knot $(m,n)$-mosaics is
a generalization of a square version of knot $n$-mosaics.

A {\em connection point\/} of a tile is defined as the midpoint of a mosaic tile edge
which is also the endpoint of a curve drawn on the tile.
Then each tile has zero, two or four connection points as illustrated in the following figure;

\begin{figure}[h]
\includegraphics[scale=1]{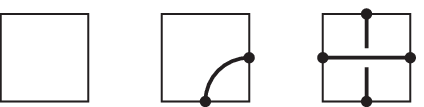}
\vspace{-2mm}
\label{fi2}
\end{figure}

Two tiles in a mosaic are called {\em contiguous\/} if they lie immediately next to each other
in either the same row or the same column.
A mosaic is said to be {\em suitably connected\/} if any pair of contiguous mosaic tiles have
or do not have connection points simultaneously on their common edge.
Note that this definition is slightly different from the original definition in \cite{LK3},
in which boundary edges of a mosaic do not have connection points.
This new definition is convenient to define a {\em quasimosaic\/} (in Section 2.)
which is suitably connected and allows connection points on boundary edges.
A {\em knot $(m,n)$-mosaic\/} is a suitably connected $(m,n)$-mosaic
whose boundary edges do not have connection points.
Then this knot $(m,n)$-mosaic represents a specific knot.
The examples of mosaics in Figure \ref{fig1} are a non-knot $(4,3)$-mosaic
and the trefoil knot $(4,4)$-mosaic.

\begin{figure}[h]
\includegraphics[scale=1]{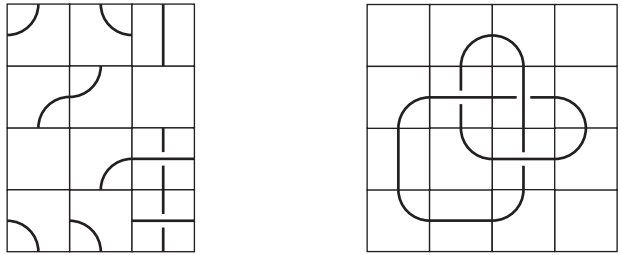}
\caption{Examples of mosaics}
\label{fig1}
\end{figure}

Let $\mathbb{K}^{(m,n)}$ denote the subset of $\mathbb{M}^{(m,n)}$ of all knot $(m,n)$-mosaics.
One of the problems in studying mosaic theory is how many knot $(m,n)$-mosaics are there.
Let $D_{m,n}$ denote the total number of elements of $\mathbb{K}^{(m,n)}$.
Indeed the original definition of $D_{n,n}$ for $m=n$ is
the dimension of the Hilbert space of quantum knot $(n,n)$-mosaics.
The main theme in this paper is to establish a table of the precise values of $D_{m,n}$
for small $m$ and $n$ by using the partition matrix argument.
Lomonaco and Kauffman \cite{LK3} showed that $D_{1,1}=1$, $D_{2,2}=2$ and $D_{3,3}=22$,
and presented a complete list of $\mathbb{K}^{(3,3)}$.

The authors \cite{HLLO1} found the precise value $D_{4,4} = 2594$,
and also a lower bound and an upper bound on $D_{n,n}$ for $n \geq 3$
which can be easily generalized to the following version for $D_{m,n}$ for $m,n \geq 3$;

\begin{theorem} \label{thm:bound} \cite{HLLO1}
For $m,n \geq 3$,
$$2^{(m-3)(n-3)} \leq \frac{275}{2(9 \cdot 6^{m-2} + 1)(9 \cdot 6^{n-2} + 1)}
\cdot D_{m,n} \leq 4.4^{(m-3)(n-3)}.$$
\end{theorem}

Also we can easily get the precise values of $D_{m,n}$ for small $m=1,2,3$.
$D_{3,n}$ is obtained from Theorem \ref{thm:bound} by applying $m=3$ directly.

\begin{corollary} \label{cor:D123}
For $m = 1,2,3$ and a positive integer $n$,
\begin{itemize}
\item $D_{1,n}=1$
\item $D_{2,n}=2^{n-1}$ for $n \geq 2$
\item $D_{3,n}=\frac{2}{5}(9 \cdot 6^{n-2} + 1)$ for $n \geq 3$
\end{itemize}
\end{corollary}

The aim of this paper is to find the precise values of $D_{m,n}$ for $m,n=4,5,6$.
Note that $D_{m,n} = D_{n,m}$.
In Section 4, we create two partition matrices
which turn out to be remarkably efficient to count small knot mosaics.

\begin{theorem} \label{thm:D456}
For $4 \leq m \leq n \leq 6$, $D_{m,n}$'s are as follows; \\

\begin{center}
\begin{tabular}{|c|r|r|r|}   \hline
$D_{m,n}$ & $n=4$ & $n=5$ & $n=6$ \\    \hline
$m=4$ & $2594$ & $54,226$ & $1,144,526$ \\    \hline
$m=5$ & & $4,183,954$ & $331,745,962$ \\    \hline
$m=6$ & & & $101,393,411,126$ \\   \hline
\end{tabular}
\end{center}

\end{theorem}

We are thankful to Lew Ludwig for introducing this problem.
Ludwig, Paat and Shapiro independently found the values of $D_{4,4}$, $D_{5,5}$ and $D_{6,6}$
by using a combination of counting techniques and computer algorithms.

Recently the authors \cite{HLLO2} announced that
they constructed an algorithm giving the precise value of $D_{m,n}$ for $m,n \geq 2$
by using a recurrence relation of matrices which are called state matrices.

Lastly we mention another natural open question related to knot mosaics
proposed by Lomonaco and Kauffman.
Define the {\em mosaic number\/} $m(K)$ of a knot $K$ as the smallest integer $n$
for which $K$ is representable as a knot $(n,n)$-mosaic.
For example, the mosaic number of the trefoil is 4 as is illustrated in Figure \ref{fig1}.
They asked ``Is this mosaic number related to the crossing number of a knot?''
The authors \cite{LHLO} established an upper bound on the mosaic number as follows;
If $K$ be a nontrivial knot or a non-split link except the Hopf link,
then $m(K) \leq c(K) + 1$.
Moreover if $K$ is prime and non-alternating except the $6^3_3$ link, then $m(K) \leq c(K) - 1$.
Note that the mosaic numbers of the Hopf link and the $6^3_3$ link are 4 and 6 respectively.

\section{Sets of quasimosaics of nine types}

A {\em quasimosaic\/} is a part of a mosaic where mosaic tiles are located at a particular
places of connected $M_{ij}$'s and these tiles are suitably connected.
A quasimosaic does not need to be rectangular.
Especially a rectangular quasimosaic is called $(p,q)$-quasimosaic
if it consists of $p$ rows and $q$ columns,
and let $\mathbb{Q}^{(p,q)}$ denote the set of all $(p,q)$-quasimosaics.
A $(p,q)$-quasimosaic is a submosaic of a knot mosaic in Lomonaco and Kauffman's definition.

An edge $e$ on a quasimosaic will be marked by ``x''
if it does not have a connection point and ``o'' if it has.
Sometimes we use a word of x and o to mark several edges together like $e_1 e_2 =$ xo
which means that edge $e_1$ does not have a connection point but edge $e_2$ has.
\vspace{3mm}

\noindent {\bf Choice rule.\/}
{\it Each $M_{ij}$ in a suitably connected mosaic has four choices $T_7$, $T_8$, $T_9$ or $T_{10}$ of mosaic tiles
if its boundary has four connection points,
and it is uniquely determined if it has zero or two connection points.
Furthermore it can not have odd number of connection points on its boundary.}
\vspace{3mm}

Now we introduce useful sets of quasimosaics of nine types named $P_1$ through
$P_9$\footnote{The same notation $P_i$'s are used for planar isotopy moves on knot mosaics
in the original Lomonaco and Kauffman's paper \cite{LK3}}.
As in Figure \ref{fig2}, let $M_*$ be some $M_{ij}$ of a given quasimosaic,
and five edges $e_1$ through $e_5$ are its typical edges in each type.
A set of quasimosaics of type $P_1$ consists of single mosaic tiles $M_*$
with the restriction on the related edges $e_1$ and $e_2$
so that $e_1 e_2 \neq$ oo (more precisely, for each fixed one among xx, xo or ox).
A set of quasimosaics of type $P_2$ is defined similarly with the condition $e_1 e_2 =$ oo.
Sets of  quasimosaics of next four types $P_3$ through $P_6$ consist of
two contiguous mosaic tiles with the restriction $e_1 e_2 \neq$ oo and $e_3 =$ x,
$e_1 e_2 \neq$ oo and $e_3 =$ o, $e_1 e_2 e_3 =$ oox, and $e_1 e_2 e_3 =$ ooo, respectively.
Sets of  quasimosaics of last three types $P_7$, $P_8$ and $P_9$ consist of
three contiguous mosaic tiles, not on the same row or the same column,
with the restriction $e_1 e_2 \neq$ oo and $e_3 e_4 \neq$ oo,
$e_1 e_2 =$ oo and $e_3 e_4 \neq$ oo (or $e_1 e_2 \neq$ oo and $e_3 e_4 =$ oo),
and $e_1 e_2 e_3 e_4 =$ oooo, respectively.
Note that this set is exhaustive.
For, there are two types for single mosaic tiles ($e_1 e_2$ is either oo or not),
four types for two contiguous mosaic tiles ($e_1 e_2$ is either oo or not, and $e_3$ is x or o)
and three types for three contiguous mosaic tiles
($e_1 e_2$ is either oo or not, and $e_1 e_2$ is either oo or not)
where type $P_8$ comprises two symmetric cases.

$|P_i|$ denotes the number of elements of a set of quasimosaics of type $P_i$ for each $i$.

\begin{figure}[h]
\includegraphics[scale=1]{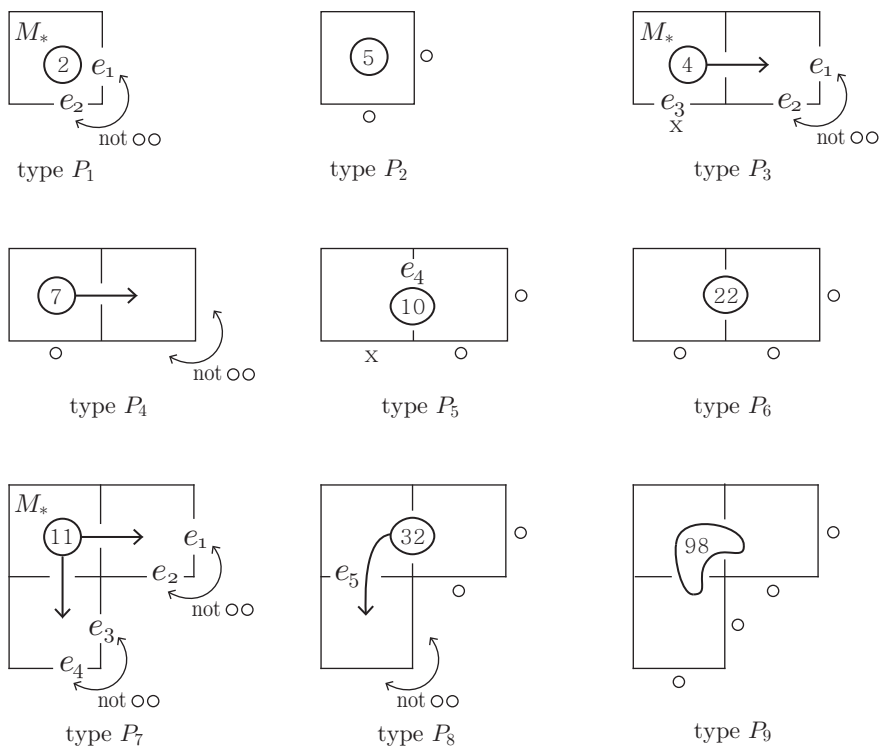}
\caption{Sets of quasimosaics of nine types. 
The notation `not oo' with an arrowed half circle, for example, in type $P_1$ means that 
$e_1 e_2$ is one among xx, xo or ox.
The numbers in circles indicate the number of choices of mosaic tiles in $M_{ij}$'s 
where the circles take possession.
We draw an arrow started at a numbered circle when the mosaic tile at the arrowhead is uniquely determined.}
\label{fig2}
\end{figure}

\begin{lemma} \label{lem:partition}
Each type $P_i$ has following values;
$|P_1|=2$, $|P_2|=5$, $|P_3|=4$, $|P_4|=7$, $|P_5|=10$, $|P_6|=22$,
$|P_7|=11$, $|P_8|=32$ and $|P_9|=98$.
\end{lemma}

\begin{proof}
$|P_1|$ and $|P_2|$ can be obtained easily from the fact that
the mosaic tile at $M_{*}$ has 2 choices among 11 mosaic tiles
if $e_1 e_2 \neq$ oo, and 5 choices if $e_1 e_2 =$ oo.

For type $P_3$ or $P_4$, the mosaic tile at $M_*$ has 4 or 7 choices
depending on $e_3 =$ x or o, respectively.
After this mosaic tile is settled, the contiguous mosaic tile
must be uniquely determined because of $e_1 e_2 \neq$ oo by Choice rule.
The arrows in the figures indicate that the mosaic tiles at arrowheads
are ``uniquely determined''.
For type $P_5$, we distinguish into two cases $e_4 =$ x or o.
In either case, the mosaic tile at $M_*$ has 2 choices,
but the contiguous mosaic tile is uniquely determined or has 4 choices, respectively.
So a set of this type has 10 kinds of quasimosaics in total.
For type $P_6$, we similarly distinguish into two cases $e_4 =$ x or o.
When $e_4 =$ x, the mosaic tile at $M_*$ has 2 choices and
the contiguous mosaic tile is uniquely determined.
When $e_4 =$ o, the mosaic tile at $M_*$ has 5 choices and
the contiguous mosaic tile has 4 choices.
So a set of this type has 22 kinds of quasimosaics.

For type $P_7$, the mosaic tile at $M_{*}$ has 11 choices,
and the two contiguous mosaic tiles are uniquely determined
after the first mosaic tile is settled.
For type $P_8$, we distinguish into two cases $e_5 =$ x or o.
In either case, the mosaic tile at the second row is uniquely determined.
But two contiguous mosaic tiles at the first row is in type $P_5$ or $P_6$
depending on $e_5 =$ x or o.
So a set of this type has 32 kinds of quasimosaics.
Finally for type $P_9$, we distinguish into two cases $e_5 =$ x or o.
When $e_5 =$ x, two contiguous mosaic tiles at the first row are in type $P_5$,
and the mosaic tile at the second row is uniquely determined.
When $e_5 =$ o, two contiguous mosaic tiles at the first row are in type $P_6$,
and the mosaic tile at the second row has 4 choices.
So a set of this type has 98 kinds of quasimosaics.
\end{proof}

\section{$D_{5,5} = 4,183,954$}

We first find the precise value of $D_{5,5}$ which can not be handled in the argument
proving the other cases.
Let $\mathbb{Q}^{(3,3)}$ denote the set of $(3,3)$-quasimosaics consisting nine mosaic tiles
at $M_{ij}$'s where $i,j=2,3,4$.
We name the interior twelve edges as in Figure \ref{fig3}.

\begin{figure}[h]
\includegraphics[scale=1]{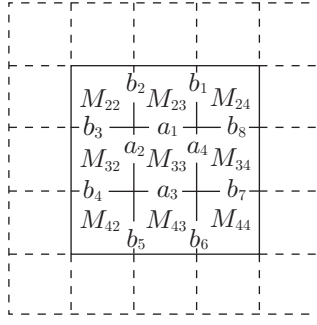}
\caption{$(3,3)$-quasimosaic}
\label{fig3}
\end{figure}

We divide into four cases according to the presences of connection points at $a_i$'s.
See Figure \ref{fig4}.

First we consider the case of $a_1 a_2 a_3 a_4 =$ xxxx.
As the first figure,
all of $M_{22}$, $M_{24}$, $M_{42}$ and $M_{44}$ are pieces of quasimosaics of type $P_7$.
We will say this briefly as $(M_{22}, M_{24}, M_{42}, M_{44})$ is of type $(P_7, P_7, P_7, P_7)$.
This means that each of $M_{22}$, $M_{24}$, $M_{42}$ and $M_{44}$ has 11 choices independently,
and then four contiguous mosaic tiles $M_{23}$, $M_{32}$, $M_{34}$ and $M_{43}$ are uniquely determined
by Choice rule.
These produce $11^4 = 14,641$ kinds of quasimosaics in total.

Now consider the case of $a_1 a_2 a_3 a_4 =$ oxox or xoxo (assume the former)
as four figures in the second row of the figure.
When $b_1 b_5 =$ xx, xo, ox and oo,
$(M_{22}, M_{24}, M_{42}, M_{44})$ is of type $(P_7, P_3, P_3, P_7)$,
$(P_7, P_3, P_4, P_8)$, $(P_8, P_4, P_3, P_7)$ and $(P_8, P_4, P_4, P_8)$ respectively.
These four occasions produce $(11^2 \cdot 4^2 + 32 \cdot 11 \cdot 7 \cdot 4 +
32 \cdot 11 \cdot 7 \cdot 4 + 32^2 \cdot 7^2) \times 2 = 143,648$ kinds of quasimosaics.
We multiplied by 2 because of the two possible choices of $a_1 a_2 a_3 a_4$.

Next consider the case of $a_1 a_2 a_3 a_4 =$ xxoo, xoox, ooxx or oxxo (assume the first one)
as four figures in the third row.
When $b_6 b_7 =$ xx, xo, ox and oo,
$(M_{22}, M_{24}, M_{42}, M_{44})$ is of type $(P_7, P_7, P_7, P_1)$,
$(P_7, P_8, P_7, P_1)$, $(P_7, P_7, P_8, P_1)$ and $(P_7, P_8, P_8, P_2)$ respectively.
These four occasions produce $(11^3 \cdot 2 + 32 \cdot 11^2 \cdot 2 +
32 \cdot 11^2 \cdot 2 + 32^2 \cdot 11 \cdot 5) \times 4 = 297,880$ kinds of quasimosaics.
We multiplied by 4 because of the four possible choices of $a_1 a_2 a_3 a_4$.

\begin{figure}[h]
\includegraphics[scale=1]{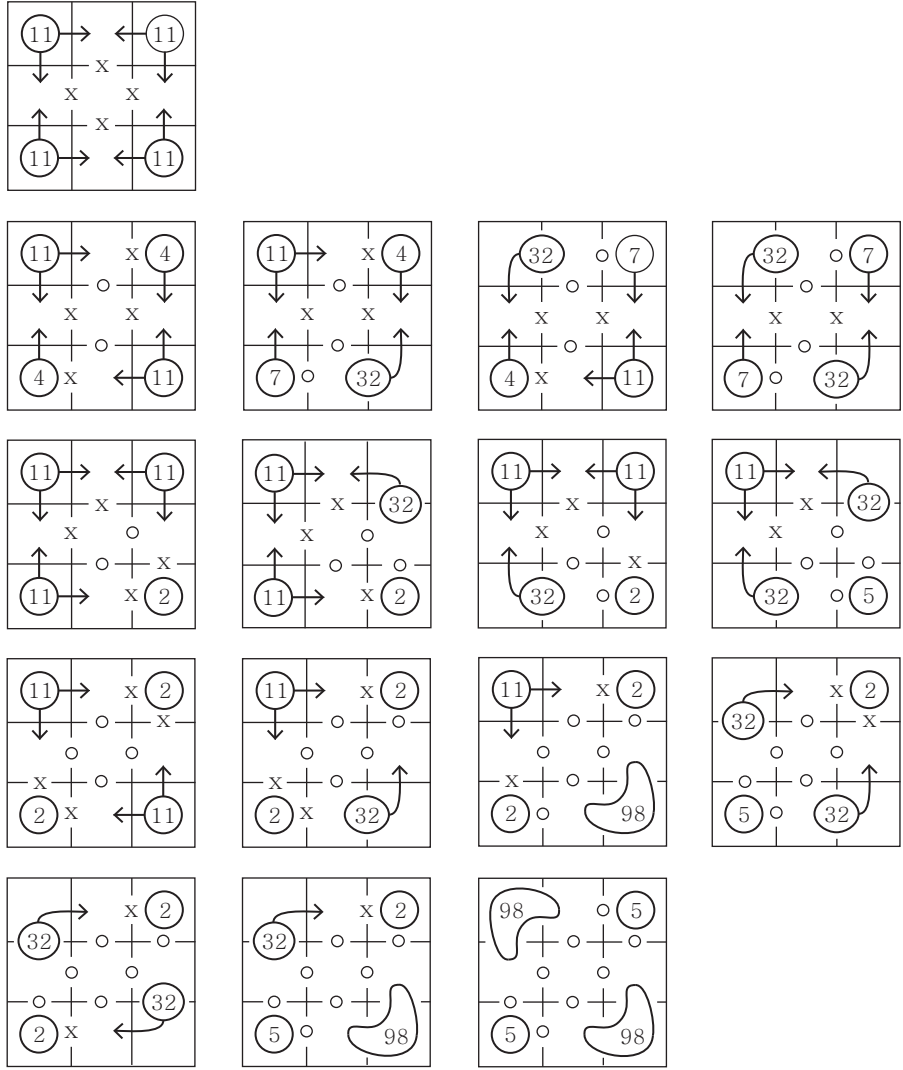}
\caption{All possible $(3,3)$-quasimosaics from four groups according to the presences of connection points 
at $a_1 a_2 a_3 a_4$ where the four groups are drawn at the first, the second, the third and the last two rows,
respectively.}
\label{fig4}
\end{figure}

Lastly consider the case of $a_1 a_2 a_3 a_4 =$ oooo as in the last two rows in the figure.
When $b_1 b_4 b_5 b_8 =$ xxxx, xxxo (and similarly for xxox, xoxx or oxxx),
xxoo (and similarly for ooxx), xoox (and similarly for oxxo), xoxo (and similarly for oxox),
xooo (and similarly for oxoo, ooxo or ooox) and oooo,
$(M_{22}, M_{24}, M_{42}, M_{44})$ has type $(P_7, P_1, P_1, P_7)$,
$(P_7, P_1, P_1, P_8)$, $(P_7, P_1, P_1, P_9)$, $(P_8, P_1, P_2, P_8)$,
$(P_8, P_1, P_1, P_8)$, $(P_8, P_1, P_2, P_9)$ and $(P_9, P_2, P_2, P_9)$ respectively.
These sixteen occasions produce $(11^2 \cdot 2^2 + 4 \cdot 32 \cdot 11 \cdot 2^2 +
2 \cdot 98 \cdot 11 \cdot 2^2 + 2 \cdot 32^2 \cdot 5 \cdot 2 + 2 \cdot 32^2 \cdot 2^2 +
4 \cdot 98 \cdot 32 \cdot 5 \cdot 2 + 98^2 \cdot 5^2) \times 4 = 1,635,808$ kinds of quasimosaics.
We multiplied by 4 because of the four possible choices of mosaic tiles of $M_{33}$
by Choice rule.

By summing all up, we got that the total number of elements of $\mathbb{Q}^{(3,3)}$ is 2,091,977.
The following rule is very useful to get knot mosaics from a quasimosaic.

\begin{lemma}[\textbf{Twofold rule}] \label{lem:two}
A $(p,q)$-quasimosaic can be extended to exactly two knot $(p+2,q+2)$-mosaics.
\end{lemma}

\begin{proof}
A $(p,q)$-quasimosaic can be extended to knot $(p+2,q+2)$-mosaics
by adjoining proper mosaic tiles surrounding it, called boundary mosaic tiles.
Since each mosaic tile has even number of connection points,
suitable connectedness guarantee that this $(p,q)$-quasimosaic has exactly
even number of connection points on its boundary.
To make a knot $(p+2,q+2)$-mosaic, all these connection points must be connected
pairwise via mutually disjoint arcs when we adjoin boundary mosaic tiles.
There are exactly two ways to do as illustrated in Figure \ref{fig5}.
\end{proof}

\begin{figure}[h]
\includegraphics[scale=1]{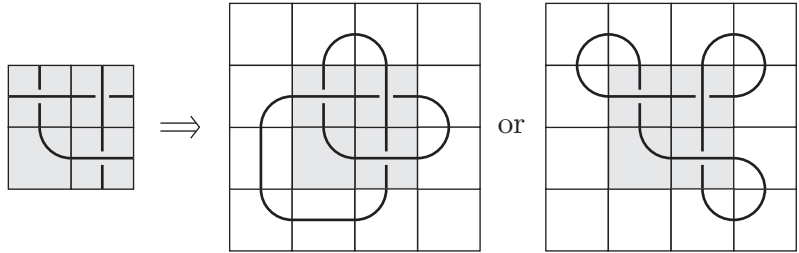}
\caption{Twofold rule}
\label{fig5}
\end{figure}

Finally we get $D_{5,5} = 4,183,954$ which is twice of the total number
of elements of $\mathbb{Q}^{(3,3)}$ by Twofold rule.

\section{Partition matrices}

A {\em partition matrix\/} $\mathbb{P}^{(p,q)}$ for the set $\mathbb{Q}^{(p,q)}$
of all $(p,q)$-quasimosaics is a $2^q \times 2^p$ matrix $(N_{ij})$
where every row (or column) is related to the presence of connection points
on the $q$ bottom (or $p$ rightmost, respectively) edges.
Roughly speaking, each $N_{ij}$ is the number of all $(p,q)$-quasimosaics
whose bottom edges and rightmost edges have specific presences of connection points
associated to the $i$-th and the $j$-th in some order, respectively.

In this section, we introduce two partition matrices $\mathbb{P}^{(1,2)}$ and $\mathbb{P}^{(2,2)}$
which would play an important role in finding the precise values of
$D_{m,n}$ for $m,n=4,5,6$ except $D_{5,5}$.
In Section 5. we build $(2,2)$-, $(2,3)$-, $(2,4)$-, $(3,4)$- and $(4,4)$-quasimosaics
(but not $(3,3)$-quasimosaics) by using $(1,2)$- and $(2,2)$-quasimosaics investigated in this section.
This construction gives the values of $D_{4,4}$, $D_{4,5}$, $D_{4,6}$, $D_{5,6}$ and $D_{6,6}$
(but not $D_{5,5}$).

First we establish a partition matrix $\mathbb{P}^{(1,2)}$ for $\mathbb{Q}^{(1,2)}$.
For an $(1,2)$-quasimosaic, we name three boundary edges on the bottom and on the right
by $b_1$, $b_2$ and $r$ as the left figure in Figure \ref{fig6}.
A partition matrix $\mathbb{P}^{(1,2)}$ is a $4 \times 2$ matrix $(N_{ij})$
where every row is related to $b_1 b_2$ and every column is related to $r$ as follows;
$N_{ij}$ is the number of all $(1,2)$-quasimosaics which have the $i$-th $b_1 b_2$
in the order of xx, xo, ox and oo,
and the $j$-th $r$ in the order of x and o.
For an example, the family of four $(1,2)$-quasimosaics for $N_{12}$
where $b_1 b_2 =$ xx and $r=$ o are illustrated on the right in Figure \ref{fig6}.
Note that the sum of all entries of $\mathbb{P}^{(1,2)}$ is
the number of elements of $\mathbb{Q}^{(1,2)}$.

\begin{figure}[h]
\includegraphics[scale=1]{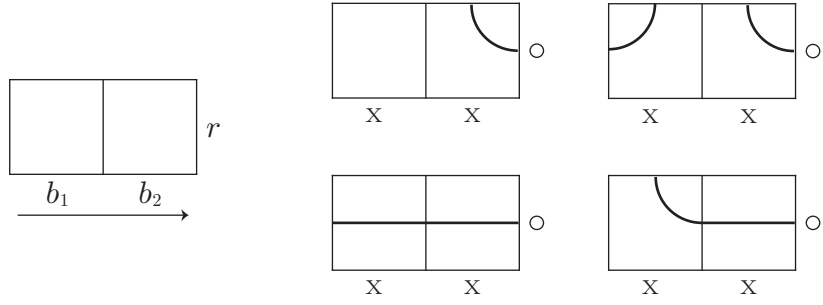}
\caption{$(1,2)$-quasimosaics and four examples for $N_{12}$ where $b_1 b_2 =$ xx and $r=$ o.}
\label{fig6}
\end{figure}

\begin{lemma} \label{lem:12}
$$\mathbb{P}^{(1,2)}=\left(
\begin{array}{cc}
 4 & 4  \\
 4 & 10  \\
 7 & 7 \\
 7 & 22  \\
\end{array}
\right)$$
\end{lemma}

\begin{proof}
The proof follows from Lemma \ref{lem:partition} directly
considering types $P_3$, $P_4$, $P_5$ and $P_6$.
\end{proof}

Next we establish another partition matrix $\mathbb{P}^{(2,2)}$ for $\mathbb{Q}^{(2,2)}$.
For an $(2,2)$-quasimosaic, we name four boundary edges on the bottom and on the right
by $b_1$, $b_2$, $r_1$ and $r_2$, and two interior edges by $c_1$ and $c_2$
as in Figure \ref{fig7}.
A partition matrix $\mathbb{P}^{(2,2)}$ is a $4 \times 4$ matrix $(N'_{ij})$
where every row is related to $b_1 b_2$ and every column is related to $r_1 r_2$ as follows;
$N'_{ij}$ is the number of all $(2,2)$-quasimosaics which have the $i$-th $b_1 b_2$,
and the $j$-th $r_1 r_2$ in the same order as previous.

\begin{figure}[h]
\includegraphics[scale=1]{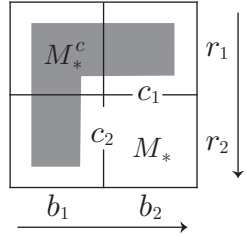}
\caption{$(2,2)$-quasimosaics}
\label{fig7}
\end{figure}

\begin{lemma} \label{lem:22}
$$\mathbb{P}^{(2,2)}=\left(
\begin{array}{cccc}
 22 & 22 & 43 & 43 \\
 22 & 55 & 43 & 139 \\
 43 & 43 & 109 & 64 \\
 43 & 139 & 64 & 403 \\
\end{array}
\right)$$
\end{lemma}

\begin{proof}
Let $M_*$ denote the bottom and the right side mosaic tile,
and $M_*^c$ the rest quasimosaic consisting of three mosaic tiles.
First, we consider the case $b_1 r_1 =$ xx for $N'_{11}$, $N'_{12}$, $N'_{21}$ and $N'_{22}$.
In this case, $M_*^c$ always has type $P_7$.
For every four choices of $b_2 r_2$, the related $c_1 c_2$ has two choices.
For example, when $b_2 r_2 =$ ox, $c_1 c_2$ must be either ox or xo.
Furthermore for each choice of $c_1 c_2$, $M_*$ is uniquely determined by Choice rule,
except when $b_2 r_2 =$ oo and  $c_1 c_2 =$ oo.
$M_*$ in the exceptional case has four choices of mosaic tiles.
Thus we get $N'_{11}$, $N'_{12}$, $N'_{21}=11 \times 2$
and $N'_{22}=11 + 11 \times 4$.

Next, consider the case $b_1 r_1 =$ xo for $N'_{13}$, $N'_{14}$, $N'_{23}$ and $N'_{24}$.
Similarly for every four choices of $b_2 r_2$, the related $c_1 c_2$ has two choices.
In this case, $M_*^c$ has type $P_7$ if $c_1 =$ x, and type $P_8$ if $c_1 =$ o.
For each choice of $c_1 c_2$, $M_*$ is uniquely determined,
except when $b_2 r_2 =$ oo and  $c_1 c_2 =$ oo, implying that $M_*$ has four choices.
Thus we get $N'_{13}$, $N'_{14}$, $N'_{23}=11 + 32$
and $N'_{24}=11 + 32 \times 4$.
The case $b_1 r_1 =$ ox for $N'_{31}$, $N'_{32}$, $N'_{41}$ and $N'_{42}$
will be handled in the same manner.

Finally, consider the case $b_1 r_1 =$ oo for the rest four entries of $\mathbb{P}^{(2,2)}$.
In this case, $M_*^c$ possibly has three types $P_7$, $P_8$ $P_9$ according to $c_1 c_2$.
And $M_*$ is uniquely determined, except when $b_2 r_2 c_1 c_2 =$ oooo.
Thus we get $N'_{33}= 11 + 98$, $N'_{34}= 32 \times 2$, $N'_{43}= 32 \times 2$
and $N'_{44}=11 + 98 \times 4$.
\end{proof}

\section{Proof of Theorem \ref{thm:D456}}

In this section, we apply partition matrices $\mathbb{P}^{(1,2)}$ and $\mathbb{P}^{(2,2)}$
to find the precise values of $D_{m,n}$ for $m,n=4,5,6$, except $D_{5,5}$.
For a matrix $\mathbb{P} = (N_{ij})$, $\|\mathbb{P}\|$ denote the sum of
all entries of $\mathbb{P}$, and $[\mathbb{P}]^2 = (N_{ij}^2)$.

Note that $\|\mathbb{P}^{(p,q)}\|$ is the total number of elements of $\mathbb{Q}^{(p,q)}$
for $(p,q)=(1,2)$ or $(2,2)$.
Furthermore each element of $\mathbb{Q}^{(p,q)}$ can be extended to exactly
two knot $(p+2,q+2)$-mosaics by Twofold rule.
Conversely, every knot $(p+2,q+2)$-mosaics can be obtained by extending a proper
$(p,q)$-quasimosaic in $\mathbb{Q}^{(p,q)}$.
Thus we can conclude $D_{m,n} = 2 \, \|\mathbb{P}^{(m-2,n-2)}\|$.
So $D_{4,4} = 2 \, \|\mathbb{P}^{(2,2)}\| = 2594$.

\subsection{\rm Partition matrix multiplying argument}   \hspace{1cm}

Let $\mathbb{P}^{(1,2)} = (N_{ij})$ and $\mathbb{P}^{(2,2)} = (N'_{ij})$.
Consider a $(2,3)$-quasimosaic $Q$ in $\mathbb{Q}^{(2,3)}$.
We name three boundary edges on the bottom by $b_1$, $b_2$ and $r'$ as upper figures
in Figure \ref{fig8}.
Let $Q_l$ and $Q_r$ be the $(2,2)$-quasimosaic obtained from the left two columns of $Q$
and the $(2,1)$-quasimosaic obtained from the rightmost column, respectively.
We name again two boundary edges of $Q_l$ on the right by $r_1$ and $r_2$,
and other two boundary edges of $Q_r$ on the left by $b'_1$ and $b'_2$.
Then $r_1 r_2$ of $Q_l$ must be the same as $b'_1 b'_2$ of $Q_r$.
Remark that $Q_l$ is an element of $\mathbb{Q}^{(2,2)}$
and $Q_r$, after rotating $90^{\circ}$ counter-clockwise, is an element of $\mathbb{Q}^{(1,2)}$.
$N_{ik}$ is the number of elements of $\mathbb{Q}^{(2,2)}$ which have the $i$-th $b_1 b_2$
and the $k$-th $r_1 r_2$ in the order of xx, xo, ox and oo,
and $N'_{kj}$ is the number of elements of $\mathbb{Q}^{(1,2)}$
which have the $k$-th $b'_1 b'_2$ in the order of xx, xo, ox and oo,
and the $j$-th $r'$ in the order of x and o.
Thus $\sum^4_{k=1} N_{ik} N'_{kj}$ is the number of elements of $\mathbb{Q}^{(2,3)}$
which have the $i$-th $b_1 b_2$ and the $j$-th $r'$.
Indeed it is the $i$-th row and the $j$-th column entry of $\mathbb{P}^{(2,2)} \cdot \mathbb{P}^{(1,2)}$.
This implies that $\| \mathbb{P}^{(2,2)} \cdot \mathbb{P}^{(1,2)} \|$ is the total number of
$\mathbb{Q}^{(2,3)}$.
Now we conclude that $D_{4,5} = 2 \, \| \mathbb{P}^{(2,2)} \cdot \mathbb{P}^{(1,2)} \| = 54,226$.

Next, consider a $(2,4)$-quasimosaic $Q'$ in $\mathbb{Q}^{(2,4)}$.
We name four boundary edges on the bottom by $b_1$, $b_2$, $r'_2$ and $r'_1$
as lower figures in Figure \ref{fig8}.
Let $Q'_l$ and $Q'_r$ be the $(2,2)$-quasimosaics obtained from the left two columns of $Q'$
and from the right two columns, respectively.
We name again four boundary edges of $Q'_l$ and $Q'_r$ as previous.
Again $r_1 r_2$ of $Q'_l$ must be the same as $b'_1 b'_2$ of $Q'_r$.
Remark that $Q'_l$ and $Q'_r$ is elements of $\mathbb{Q}^{(2,2)}$.
Similarly we rotate $Q'_r$ $90^{\circ}$ counter-clockwise.
Thus $\sum^4_{k=1} N_{ik} N_{kj}$ is the number of elements of $\mathbb{Q}^{(2,4)}$
which have the $i$-th $b_1 b_2$ and the $j$-th $r'_1 r'_2$.
This implies that $D_{4,6} = 2 \, \| \mathbb{P}^{(2,2)} \cdot \mathbb{P}^{(2,2)} \| = 1,144,526$.

\begin{figure}[h]
\includegraphics[scale=1]{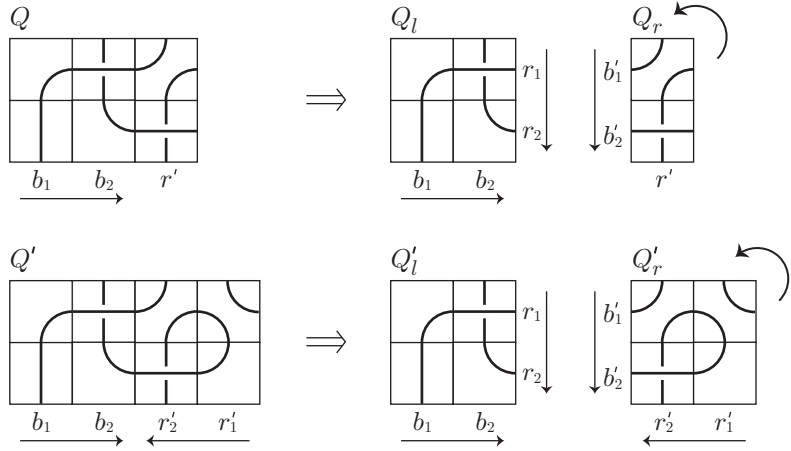}
\caption{Dividing a $(2,3)$-quasimosaic and a $(2,4)$-quasimosaic 
into $(2,1)$-quasimosaics and $(2,2)$-quasimosaics.}
\label{fig8}
\end{figure}

\subsection{\rm Partition matrix squaring argument}   \hspace{1cm}

Let $\mathbb{P}^{(2,2)} \cdot \mathbb{P}^{(1,2)} = (N_{ij})$.
Consider a $(4,3)$-quasimosaic $Q$ in $\mathbb{Q}^{(4,3)}$.
We name three interior edges on the middle by $b_1$, $b_2$ and $r$ as upper figures
in Figure \ref{fig9}.
Let $Q_u$ and $Q_l$ be the $(2,3)$-quasimosaics obtained from the upper two rows of $Q$
and from the lower two rows, respectively.
We name again three boundary edges of $Q_u$ on the bottom by $b'_1$, $b'_2$ and $r'$,
and other three boundary edges of $Q_l$ on the top by $b''_1$, $b''_2$ and $r''$,
so that $b'_1 b'_2 r' = b''_1 b''_2 r''$.
We reflect $Q_l$ through a horizontal line.
Remark that $Q_u$ and $Q_l$ is elements of $\mathbb{Q}^{(2,3)}$.
$N_{ij}$ is the number of elements of $\mathbb{Q}^{(2,3)}$ which have the $i$-th $b'_1 b'_2$
and the $j$-th $r'$.
Thus $N_{ij}^2$ is the number of elements of $\mathbb{Q}^{(4,3)}$
which have the $i$-th $b_1 b_2$ and the $j$-th $r$.
This implies that
$D_{5,6} = D_{6,5} = 2 \, \|[ \mathbb{P}^{(2,2)} \cdot \mathbb{P}^{(1,2)}]^2 \| = 331,745,962$.

Now let $\mathbb{P}^{(2,2)} \cdot \mathbb{P}^{(2,2)} = (N_{ij})$.
Consider a $(4,4)$-quasimosaic $Q$ in $\mathbb{Q}^{(4,4)}$.
We name four interior edges on the middle by $b_1$, $b_2$, $r_2$ and $r_1$ as lower figures
in Figure \ref{fig9}.
The similar argument as previous guarantees that $N_{ij}^2$ is the number of elements
of $\mathbb{Q}^{(4,4)}$ which have the $i$-th $b_1 b_2$ and the $j$-th $r_1 r_2$.
This implies that
$D_{6,6} = 2 \, \|[ \mathbb{P}^{(2,2)} \cdot \mathbb{P}^{(2,2)}]^2 \| = 101,393,411,126$.

\begin{figure}[h]
\includegraphics[scale=1]{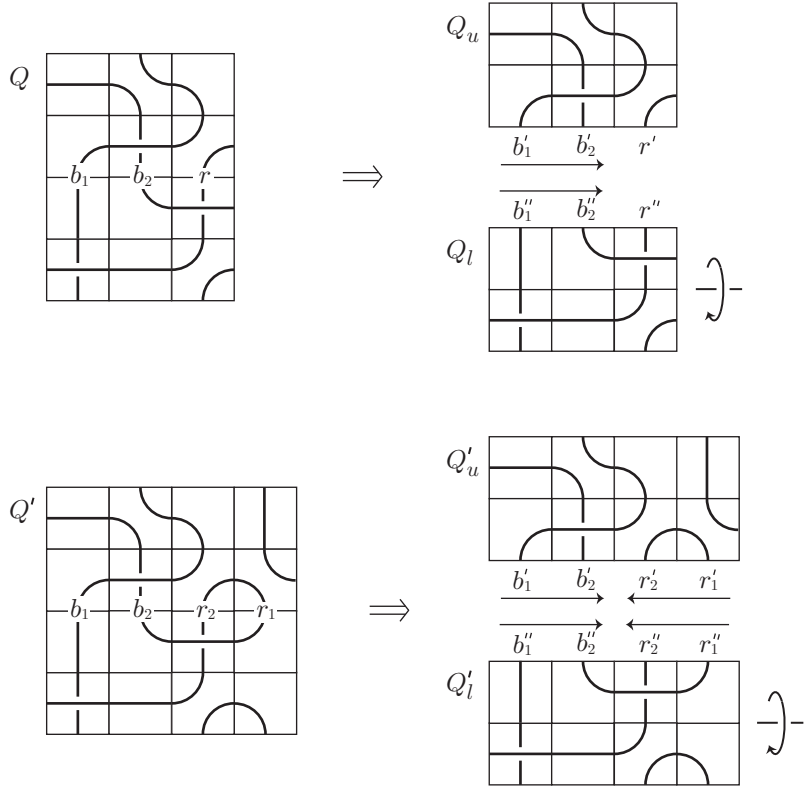}
\caption{Dividing a $(4,3)$-quasimosaic and a $(4,4)$-quasimosaic 
into $(2,3)$-quasimosaics and $(2,4)$-quasimosaics, respectively.}
\label{fig9}
\end{figure}

\section{Conclusion}

In this paper, we found the cardinality of knot $(m,n)$-mosaics $D_{m,n}$ for $m,n=4,5,6$.
Mainly we build sets of quasimosaics of nine types
to calculate two partition matrices related to $(1,2)$- and $(2,2)$-quasimosaics.
These partition matrices turn out to be remarkably efficient to count small knot mosaics,
even though $D_{m,n}$ increases rapidly so that $D_{6,6}$ is larger than $10^{11}$.
In Section 5. we introduce partition matrix multiplying argument and squaring argument
to find $D_{m,n}$ for only $m,n=4,5,6$ from these two partition matrices.
Eventually, if we have partition matrices for bigger quasimosaics,
then we can find the values of $D_{m,n}$ for larger $m,n$ by applying the arguments here.
For example, partition matrices for $(1,3)$-, $(2,3)$- and $(3,3)$-quasimosaics
is enough to calculate $D_{m,n}$ for $m,n=4,5,6,7,8$.
Finding these partition matrices for bigger quasimosaics are worthy of further research.

Recently the authors \cite{HLLO2} constructed an algorithm producing $D_{m,n}$ for positive $m,n \geq 2$
by using a recurrence relation of matrices which are called state matrices.
State matrices are a generalized version of partition matrices.

\end{document}